\documentclass[12pt,twoside]{amsart}

\usepackage{amssymb}
\usepackage{graphicx}
\usepackage[pdfauthor={David Cook II},
            pdftitle={The strong Lefschetz property in codimension two},
            pdfsubject={Commutative algebra},
            pdfkeywords={Strong Lefschetz property, positive characteristic, lexsegment ideals},
            pdfproducer={LaTeX with hyperref},
            pdfcreator={latex->dvips->ps2pdf},
            pdfborder={0 0 0},
            colorlinks=false]{hyperref}
\usepackage[margin=1in]{geometry}

\newcommand{\NN}{\mathbb{N}}
\newcommand{\PS}{\mathfrak{S}}
\newcommand{\ZZ}{\mathbb{Z}}
\DeclareMathOperator{\charf}{char}
\DeclareMathOperator{\Hilb}{Hilb}
\DeclareMathOperator{\gin}{gin}
\DeclareMathOperator{\inid}{in}
\DeclareMathOperator{\indeg}{indeg}

\DeclareMathOperator{\reg}{reg}
\DeclareMathOperator{\sgn}{sgn}
\newcommand{\st}{\; | \;}                                     
\def\urltilda{\kern -.15em\lower .7ex\hbox{\~{}}\kern .04em}  

\numberwithin{figure}{section}
\numberwithin{equation}{section}

\newtheorem{theorem}{Theorem}[section]
\newtheorem{lemma}[theorem]{Lemma}
\newtheorem{proposition}[theorem]{Proposition}
\newtheorem{corollary}[theorem]{Corollary}
\newtheorem{conjecture}[theorem]{Conjecture}

\theoremstyle{definition}

\newtheorem{remark}[theorem]{Remark}
\newtheorem{example}[theorem]{Example}

\newtheorem*{acknowledgement}{Acknowledgement}

\begin{document}

\title{The strong Lefschetz property in codimension two}
\author[D.\ Cook II]{David Cook II}
\address{Department of Mathematics, University of Notre Dame, Notre Dame, IN 46556, USA}
\email{\href{mailto:dcook8@nd.edu}{dcook8@nd.edu}}

\begin{abstract}
    Every artinian quotient of $K[x,y]$ has the strong Lefschetz property if $K$
    is a field of characteristic zero or is an infinite field whose characteristic is
    greater than the regularity of the quotient.  We improve this bound in the case
    of monomial ideals.  Using this we classify when both bounds are sharp.
    Moreover, we prove that the artinian quotient of a monomial ideal in $K[x,y]$ always
    has the strong Lefschetz property, regardless of the characteristic of the field, exactly when
    the ideal is lexsegment.  As a consequence we describe a family of non-monomial complete
    intersections that always have the strong Lefschetz property.
\end{abstract}

\maketitle

\section{Introduction}\label{sec:intro}

Let $K$ be an infinite field of arbitrary characteristic, and let $I$ be a homogeneous artinian ideal in $S = K[x_1, \ldots, x_n]$.
The quotient $S/I$ is said to have the \emph{strong Lefschetz property} if there exists a linear form $\ell \in [S/I]_1$ such that for
all integers $d \geq 0$ and $t \geq 1$ the map $\times \ell^t : [S/I]_d \rightarrow [S/I]_{d+t}$ has maximal rank.  In this
case, $\ell$ is called a \emph{strong Lefschetz element} of $S/I$.  If the maps have maximal rank for $t = 1$, then $S/I$ is said to have
the \emph{weak Lefschetz property}, and $\ell$ is called a \emph{weak Lefschetz element} of $S/I$.

The Lefschetz properties have been studied extensively; see the recent survey by Migliore and Nagel~\cite{MN} and the references
contained therein.  The interest in these properties largely stems from constraints on the Hilbert
functions of quotients that have the weak or strong Lefschetz property (see, e.g., \cite{BMMNZ, HMNW, MZ}).

Until recently, most results have focused on characteristic zero or on at least three variables.  For artinian quotients of
$K[x,y]$, this is not without reason:  the weak Lefschetz property always holds, regardless of characteristic.
This was explicitly proven for characteristic zero by Harima, Migliore, Nagel, and Watanabe in~\cite[Proposition~4.4]{HMNW}
(see the note following the next theorem for more on the characteristic zero case).  It was proven for arbitrary characteristic 
by Migliore and Zanello in~\cite[Corollary~7]{MZ}, though it was not specifically stated therein as noted by Li and Zanello 
in~\cite[Remark~2.6]{LZ} (see also~\cite{CN}).

\begin{theorem}{\cite[Corollary~7]{MZ}} \label{thm:wlp}
    Every artinian ideal in $K[x,y]$ has the weak Lefschetz property, regardless of the characteristic of $K$.
\end{theorem}

Further still, the strong Lefschetz property is known to hold when the characteristic is zero or greater
than the regularity of the quotient.

\begin{theorem} \label{thm:slp-char0-reg}
    Let $I$ be a homogeneous artinian ideal in $R = K[x,y]$, where $K$ is a field of characteristic $p \geq 0$.
    Then $R/I$ has the strong Lefschetz property if $p = 0$ or $p > \reg{R/I}$.
\end{theorem}

This result has a varied history.  The characteristic zero part was first explicitly given by Harima, Migliore, Nagel, 
and Watanabe~\cite[Proposition~4.4]{HMNW}.  Their proof relies on the generic initial ideal being strongly-stable.  Recall
that the generic initial ideal is strongly-stable in characteristic zero but also in characteristics larger than the 
largest exponent of a minimal generator of the ideal (see, e.g., \cite[Proposition~4.2.4(b)]{HH}).
Hence the proof of~\cite[Proposition~4.4]{HMNW} extends to the positive characteristic restriction given above.  Using a different 
approach, Basili and Iarrobino proved a much stronger result~\cite[Theorem~2.16]{BI} which reduces to the theorem as 
stated above.  Further still, Iarrobino has pointed out to us that the characteristic zero part follows from a much 
earlier result of Brian\c{c}on~\cite{Br} and the positive characteristic part follows from an earlier result of 
his own~\cite[Theorem~2.9]{Ia}.

In this paper we consider the presence of the strong Lefschetz property for homogeneous artinian
quotients of $R = K[x,y]$, where the characteristic of $K$ is positive.  In Section~\ref{sec:width} we recall some needed
definitions and introduce the width function of a monomial ideal.  The possible width functions are classified
in Proposition~\ref{pro:classify-width}, which is analogous to Macaulay's Theorem for Hilbert functions.
In Section~\ref{sec:max-rank}, we derive conditions to determine when the multiplication map
$\times \ell^t: [R/I]_d \rightarrow [R/I]_{d+t}$ has maximal rank for monomial quotients of $R$.

Section~\ref{sec:slp} contains the main results of this paper.  In particular, Theorem~\ref{thm:w-bound} bounds
the characteristics in which the strong Lefschetz property can be absent from monomial quotients by means of the
width function.  From this we recover Theorem~\ref{thm:slp-char0-reg} using different techniques than used in~\cite{HMNW} and~\cite{BI}.
Furthermore, we classify when the bounds in Theorem~\ref{thm:w-bound} and Theorem~\ref{thm:slp-char0-reg} are sharp in
Corollary~\ref{cor:w-sharpness} and Corollary~\ref{cor:p-failure-reg-sharpness}, respectively.  In Theorem~\ref{thm:always-slp}, 
we show that a monomial quotient \emph{always} has the strong Lefschetz property if and only if it is an artinian quotient
of a lexsegment ideal.

We close with some observations in Section~\ref{sec:obs}.  We use Proposition~\ref{pro:slp-small-nonmono} to show
that there exist non-monomial complete intersections that always have the strong Lefschetz property, and thus the presence
of the strong Lefschetz property for complete intersections is \emph{not} determined by the ci-type.  In
Sections~\ref{sub:nilps} and~\ref{sub:codim3}, we briefly describe connections to enumerative combinatorics
and the weak Lefschetz property in codimension three, respectively.

Throughout the remainder of this paper $R = K[x,y]$, where $K$ is an infinite field of characteristic $p \geq 0$.

\section{The width function}\label{sec:width}

Let $I$ be a homogeneous ideal of $S = K[x_1, \ldots, x_n]$.  Recall that each component $[S/I]_d$ is a finite dimensional
$K$-vector space, and the \emph{Hilbert function} of $S/I$ is the function $h_{S/I}: \NN_0  \to \NN_0$, where
$h(d) := h_{S/I}(d) := \dim_K [S/I]_{d}$.  If there is an integer $r$ such that $h(i) > 0$ if and only if
$0 \leq i \leq r$, then $S/I$ is said to be \emph{artinian}; in this case, $r$ is the \emph{regularity} of $S/I$ and
is denoted $\reg{S/I}$.  If $S/I$ is artinian and $r = \reg{S/I}$, then we call the finite sequence $(h(0), \ldots, h(r))$,
where $h = h_{S/I}$, the \emph{$h$-vector} of $S/I$.  Further still, the \emph{initial degree} of $I$ is the smallest
degree of a minimal generator of $I$ and is denoted $\indeg{I}$.  Thus $[S/I]_i \cong [S]_i$ for $0 \leq i < \indeg{I}$.

\subsection{Lexsegment ideals \& Macaulay's Theorem}\label{sub:lexseg}~

Suppose $x_1 > \cdots > x_n$ in $S$.  The \emph{lexicographic order} on the monomials in $S$ is given by
$x_1^{a_1} \cdots x_n^{a_n} > x_1^{b_1} \cdots x_n^{b_n}$ if either $\sum_{i=1}^n a_i > \sum_{i=1}^n b_i$
or $\sum_{i=1}^n a_i = \sum_{i=1}^n b_i$ and the leftmost nonzero component of the vector
$(b_1, \ldots, b_n) - (a_1, \ldots, a_n)$ is negative.  On the other hand, the \emph{reverse lexicographic
order} on the monomials in $S$ is given by $x_1^{a_1} \cdots x_n^{a_n} > x_1^{b_1} \cdots x_n^{b_n}$ if either
$\sum_{i=1}^n a_i > \sum_{i=1}^n b_i$ or $\sum_{i=1}^n a_i = \sum_{i=1}^n b_i$ and the rightmost nonzero component of the vector
$(b_1, \ldots, b_n) - (a_1, \ldots, a_n)$ is positive.  Notice that in the case of two variables ($n=2$), these
two orders are the same.

A monomial ideal $I$ in $S$ is \emph{lexsegment in degree $d$} if for any two monomials $u, v \in [I]_d$ and any
monomial $m \in [S]_d$ such that $u \leq m \leq v$ in the lexicographic order, then $m \in I$.  If $I$ is lexsegment
in every degree, then $I$ is said to be a \emph{(completely) lexsegment ideal}.  Further, a lexsegment ideal $I$ is
called an \emph{initial lexsegment ideal} if $x_1^d \in I$ for every degree $d$ such that $[I]_d \neq 0$.

In order to state Macaulay's Theorem (see, e.g., \cite[Theorem~6.3.8]{HH}), we must first define some notation.
Let $a$ and $d$ be positive integers.  The \emph{Macaulay expansion} of $a$ with respect to $d$ is the unique
expansion $a = \binom{a_d}{d} + \cdots + \binom{a_k}{k}$, where $a_d > \cdots > a_k \geq k \geq 1$.  Further,
we define $a^{\langle d \rangle} = \binom{a_d+1}{d+1} + \cdots + \binom{a_k+1}{k+1}$, and we set $0^{\langle d \rangle} = 0$
for all positive integers $d$.

\begin{theorem}[Macaulay's Theorem]\label{thm:mac}
    Let $h: \NN_0 \to \NN_0$ be a function.  The following statements are equivalent:
    \begin{enumerate}
    \item $h$ is the Hilbert function of a standard graded $K$-algebra,
    \item $h$ is the Hilbert function of an initial lexsegment quotient in $h(1)$ variables, and
    \item $h(0) = 1$ and $h(d+1) \leq h(d)^{\langle d \rangle}$ for all $d \geq 1$.
    \end{enumerate}
\end{theorem}

This allows us to immediately classify the Hilbert functions of ideals in two variables.

\begin{proposition}\label{pro:classify-hf}
    Let $h: \NN_0 \to \NN_0$ be a function.  Then $h$ is the Hilbert function of some (proper) homogeneous
    quotient in $R = K[x,y]$ if and only if there exists a nonnegative integer
    $d$ so that $h(j) = j+1$ for $0 \leq j \leq d$ and $h(j) \geq h(j+1) \geq 0$ for all $j \geq d$.
\end{proposition}

Moreover, we classify the Hilbert functions that force a monomial ideal to be lexsegment.

\begin{lemma}\label{lem:h-force-lexsegment}
    Suppose $h: \NN_0 \to \NN_0$ is the Hilbert function of some quotient of $R = K[x,y]$.
    Then every monomial quotient $R/I$ with $h_{R/I} = h$ is a lexsegment quotient if and only if for every nonnegative
    integer $d$ such that $h(d) > h(d+1)$ then $h(d+1) = h(d+2)$.
\end{lemma}
\begin{proof}
    We prove the negation of the desired statement in two parts.  Moreover, all comparisons of monomials in
    $R = K[x,y]$ are in the lexicographic order with $x > y$, and all ordered sets are presented in ascending order.

    Suppose that there exists a nonnegative integer $d$ such that $h(d) > h(d+1)$ and $h(d+1) \neq h(d+2)$.  By
    Proposition~\ref{pro:classify-hf}, once a Hilbert function is weakly decreasing, then it must remain so;
    hence $h(d+1) > h(d+2)$.  Let $I$ be the initial lexsegment ideal with Hilbert function $h$, as guaranteed
    by Macaulay's Theorem (see Theorem~\ref{thm:mac}).  By construction, $[I]_d$ is spanned by $a = d+1 - h(d)$ monomials
    of degree $d$, in particular, by the set $A = \{x^{d-a+1}y^{a-1}, \ldots, x^d\}$.  Similarly, $[I]_{d+1}$ and
    $[I]_{d+2}$ are spanned by $b = d+2 - h(d+1)$ monomials of degree $d+1$ and $c = d+3 - h(d+2)$ monomials of degree $d+2$,
    respectively; let $B = \{x^{d-b+2}y^{b-1}, \ldots, x^{d+1}\}$ and $C = \{x^{d-c+3}y^{c-1}, \ldots, x^{d+2}\}$ be those monomials.

    Let $B' = \{x^{d-b+1}y^{b}, x^{d-b+3}y^{b-2}, \ldots, x^{d+1}\}$.  Since $h(d) > h(d+1)$, we have that $b-a \geq 2$
    and so every product of a member of $A$ with either $x$ or $y$ is in $B'$; in particular, $x^{d-a+1}y^{a-1} \cdot y$
    is in $B'$ since $a \leq b-2$.  Further, since $h(d+1) > h(d+2)$, we have that $c-b \geq 2$.  Hence every product of
    a member of $B'$ with either $x$ or $y$ is in $C$; in particular, $x^{d-b+1}y^{b} \cdot y$ is in $C$ since $b+1 \leq c-1$.
    Let $J$ be given by $[J]_i = [I]_i$ for $i \neq d+1$ and suppose $[J]_{d+1}$ spanned by $B'$. Then $J$ is an ideal
    with Hilbert function $h_{R/J} = h$ but is not lexsegment in degree $d+1$.

    Now suppose that there exists a monomial quotient $R/I$ with $h_{R/I} = h$ that is not lexsegment.  That is, there
    is a degree, say, $d+1$, such that $[I]_{d+1}$ is spanned by $d+2 - h(d+1)$ monomials such that there are at most
    $d - h(d+1)$ pairs of consecutive monomials.  For every monomial $m$ in $[I]_d$, $\{ym, xm\}$ forms a consecutive
    pair in $[I]_{d+1}$, so there can be at most $d-h(d+1)$ monomials spanning $[I]_d$.  Since exactly $d+1 - h(d)$
    monomials span $[I]_d$, then $h(d+1) + 1 \leq h(d)$.  Moreover, for every monomial $m'$ in $[I]_{d+1}$, the monomials
    $xm'$ and $ym'$ are in $[I]_{d+2}$.  However, consecutive monomials overlap in exactly one multiple.  Hence
    there are at least $2(d+2 - h(d+1)) - (d - h(d+1)) = d + 4 - h(d+1)$ monomials in $[I]_{d+2}$.  Since exactly
    $d+3 - h(d+2)$ monomials span $[I]_{d+2}$, then $h(d+2) + 1 \leq h(d+1)$.  Thus we have that
    $h(d) > h(d+1) > h(d+2)$.
\end{proof}

\subsection{The width function of a monomial ideal}\label{sub:width-func}~

Throughout this section, all comparisons of monomials in $R = K[x,y]$ are in the lexicographic order with $x > y$.

Let $I$ be a (not necessarily artinian or lexsegment) monomial ideal of $R$.  The \emph{width function} of $R/I$ is
the function $w_{R/I}: \NN_0 \to \NN_0$ defined as follows.
If $0 \leq d < \indeg{I}$, then $w_{R/I}(d) = 0$.  Suppose $d \geq \indeg{I}$.  Let $b$ be the smallest integer so
that $x^b y^{d-b} \in I$, and let $c$ be the largest integer so that $x^c y^{d-c} \in I$.  Then $w_{R/I}(d) = c-b+1$
is the ``width'' of the monomials in $[I]_d$.  If $R/I$ is artinian and $r = \reg{R/I}$, then we call the
finite sequence $(w(0), \ldots, w(r))$, where $w = w_{R/I}$, the \emph{$w$-vector} of $R/I$.

\begin{example}
    Let $I = (x^6, x^3y, xy^5, y^5)$.  Then the $h$-vector of $R/I$ is $(1, 2, 3, 4, 4, 2)$ and the
    $w$-vector of $R/I$ is $(0,0,0,0,1,5)$.  Notice that in degree $4$, the only monomial in $[I]_4$
    is $x^3y$, hence $w_{R/I}(4) = 1$.  However, in degree $5$, the monomials in $[I]_5$ are
    $y^5, xy^4, x^3y^2,$ and $x^4y$, so $w_{R/I}(5) = 5$.
\end{example}

We now classify the possible width functions of monomial ideals.

\begin{proposition}\label{pro:classify-width}
    Let $w: \NN_0 \to \NN_0$ be a function, and let $R = K[x,y]$.  The following statements are equivalent:
    \begin{enumerate}
    \item $w = w_{R/I}$, where $I$ is a monomial ideal,
    \item $w = w_{R/I}$, where $I$ is an initial lexsegment ideal,
    \item $w(d) = d+1 - h_{R/I}(d)$, where $I$ is an initial lexsegment ideal, and
    \item either $w(d) = d+1$ for all $d \geq 0$ or there exists an integer $m > 0$ so
        that $w(d) = 0$ for $d < m$ and $1 \leq w(d) < w(d+1) \leq d+2$, for $d \geq m$.
    \end{enumerate}
\end{proposition}
\begin{proof}
    Clearly (ii) implies (i).  We proceed by showing (i) $\Rightarrow$ (iv) $\Rightarrow$ (iii)
    $\Rightarrow$ (ii).

    \emph{(i) $\Rightarrow$ (iv)}: Let $I$ be a monomial ideal with width function $w = w_{R/I}$.
    If $I = R$, then $w(d) = d+1$ for all $d$.  Suppose $I \neq R$, then $\indeg{I} \geq 1$.
    By construction, $w(d) = 0$ for $d < \indeg{I}$.  Let $d \geq \indeg{I}$, and set
    $b = \min\{ i \st x^i y^{d-i} \in I\}$ and $c = \max\{ i \st x^i y^{d-i} \in I\}$, i.e., $w(d) = c-b+1$.
    Clearly $0 \leq b \leq c \leq d$, so $1 \leq w(d) \leq d+1$.  Moreover, since $x^b y^{d-b}$
    and $x^c y^{d-c}$ are both in $I$, then $x^b y^{d+1-b}$ and $x^{c+1} y^{d-c}$ are both in $I$.
    Thus $b \geq \min\{ i \st x^i y^{d+1-i}\}$ and $c+1 \leq \max\{ i \st x^i y^{d+1-i} \in I\}$,
    and so $w(d+1) \geq c+1-b+1 > w(d)$.

    \emph{(iv) $\Rightarrow$ (iii)}:  If $w(d) = d+1$ for all $d \geq 0$, then $w(d) = d+1 - h_{R/I}(d)$,
    where $I = R$, which is clearly an initial lexsegment ideal.  Suppose now $w: \NN_0 \to \NN_0$
    is a function such that there exists an integer $m > 0$ so that $w(d) = 0$ for $d < m$ and
    $1 \leq w(d) < w(d+1) \leq d+2$, for $d \geq m$.  Define $h: \NN_0 \to \NN_0$ by
    $h(d) = d+1 - w(d)$ for $d \geq 0$.  Thus $h(0) = 1$, $h(1) \leq 2$, and $h(d) = d+1$ for
    $0 \leq d < m$.  Moreover, since $w(d) < w(d+1) \leq d+2$ for $d \geq m$, then
    $h(d) = d+1 - w(d) \geq d+2 - w(d+1) = h(d+1) \geq 0$.  Hence by Proposition~\ref{pro:classify-hf},
    $h$ is a Hilbert function of some proper homogeneous quotient of $R$.  Thus by Macaulay's
    Theorem (see Theorem~\ref{thm:mac}), there exists an initial lexsegment ideal $I$ such that
    $h = h_{R/I}$.

    \emph{(iii) $\Rightarrow$ (ii)}: Let $w(d) = d+1 - h_{R/I}(d)$, where $I$ is an initial lexsegment ideal.
    As $I$ is lexsegment in degree $d$, for all $d$, then there are $d+1 - h_{R/I}(d) = w(d)$ monomials
    of degree $d$ in $I$.  Moreover, these $w(d)$ monomials are consecutive in the lexicographic order
    and so $w_{R/I}(d) = w(d)$.

    Thus the four statements (i)--(iv) are indeed equivalent.
\end{proof}

Further, we classify the width functions that force a monomial ideal to be lexsegment.

\begin{lemma}\label{lem:w-force-lexsegment}
    Suppose $w: \NN_0 \to \NN_0$ is the width function of some monomial quotient of $R = K[x,y]$.
    Then every monomial quotient $R/I$ with $w_{R/I} = w$ is lexsegment if and only if for every
    nonnegative integer $d$ we have $0 \leq w(d+1) - w(d) \leq 2$.
\end{lemma}
\begin{proof}
    We prove the negation of the desired statement in two parts.  Moreover, all comparisons of monomials in
    $R = K[x,y]$ are in the lexicographic order with $x > y$, and all ordered sets are presented in ascending order.

    Suppose that there exists a nonnegative integer $d$ so that $w(d+1) - w(d) > 2$.  Let $I$ be the initial
    lexsegment ideal with width function $w_{R/I} = w$, as guaranteed by Proposition~\ref{pro:classify-width}.
    Hence $[I]_{d+1}$ is spanned by the $w(d+1)$ monomials $B = \{x^{d-w(d+1)+2}y^{w(d+1)-1}, \ldots, x^{d+1}\}$.
    Notice that the next to last monomial of $B$, namely $x^{d-w(d+1)+3}y^{w(d+1)-2}$, is \emph{not} a multiple
    of a monomial in $[I]_d$, since $x^{d-w(d)+1}y^{w(d)-1}$ is the smallest monomial in $[I]_d$ and $w(d) \leq w(d+1)-3$.

    Let $B'$ be $B \setminus \{x^{d-w(d)+1}y^{w(d)-1}\}$.  Hence if $J$ given by $[J]_i = [I]_i$ for $i \neq d+1$ and
    $[J]_{d+1}$ spanned by $B'$, then $J$ is an ideal with width function $w_{R/J} = w$, and $J$ is not lexsegment
    in degree $d+1$.

    Now suppose that there exists a monomial quotient $R/I$ with $w_{R/I} = w$ that is lexsegment in every
    degree $i \leq d$ but not lexsegment in degree $d+1$.  Thus there are $w(i)$ consecutive monomials in $I$ of degree $i$
    for $i \leq d$, and there are less than $w(d+1)$ monomials in $I$ of degree $d+1$.  By Proposition~\ref{pro:classify-width}
    we have $w(d) < w(d+1)$.  Since $I$ is not lexsegment in degree $d+1$, then there must be at least one monomial that is
    not consecutive to one of the $w(d)+1$ consecutive multiples of the $w(d)$ consecutive monomials in degree $d$.  Thus the
    width in degree $d+1$ must be at least two larger than $w(d)+1$ (one for the absent monomial and one for the
    guaranteed non-consecutive monomial).  That is, $w(d+1) \geq w(d)+3$.
\end{proof}

For a monomial ideal $I$ of $R$, the number of degree $d$ monomials not in $I$ between the
smallest and largest degree $d$ monomials in $I$ can be derived easily from the width function and the Hilbert function.
This number can be thought of as the ``lexsegment defect'' of $I$ in degree $d$.

\begin{lemma} \label{lem:lex-seg-dist}
    Let $I$ be a monomial ideal in $R$, where the monomials of $R$ are ordered lexicographically.
    The number of degree $d$ monomials not in $I$ between the smallest and largest degree $d$ monomials in $I$ is
    $w_{R/I}(d) + h_{R/I}(d) - (d+1)$, which is zero if and only if $I$ is lexsegment in degree $d$.
\end{lemma}
\begin{proof}
    For $d \geq 0$, there are $d+1-h_{R/I}(d)$ degree $d$ monomials \emph{not} in $I$, thus there
    are $w_{R/I}(d) - (d+1 - h_{R/I}(d))$ degree $d$ monomials \emph{not} in $I$ between the
    smallest and largest degree $d$ monomials in $I$.  Furthermore, $I$ is lexsegment in degree $d$ if and only
    if there are no missing monomials between the smallest and largest degree $d$ monomials in $I$.
\end{proof}

\section{Maximal rank multiplication maps}\label{sec:max-rank}~

Throughout this section, all comparisons of monomials in $R = K[x,y]$ are in the lexicographic order with $x > y$,
and all ordered sets are presented in ascending order.

Recall that a monomial algebra has the weak (strong) Lefschetz property exactly when the sum of the variables
is a weak (strong) Lefschetz element.

\begin{proposition}{\cite[Proposition~2.2]{MMN}} \label{pro:mono}
    Let $I$ be a monomial artinian ideal in the ring $S = K[x_1, \ldots, x_n]$.  Then $S/I$ has the weak (strong)
    Lefschetz property if and only if $x_1 + \cdots + x_n$ is a weak (strong) Lefschetz element of $S/I$.
\end{proposition}

In particular, Theorem~\ref{thm:wlp} immediately implies that we need only look at the maps between
equi-dimensional components.

\begin{lemma} \label{lem:slp-same-hf}
    Let $I$ be a monomial artinian ideal in $R$.  Then $R/I$ has the strong Lefschetz property
    if and only if the map $\times (x+y)^t: [R/I]_d \rightarrow [R/I]_{d+t}$ is a bijection for all integers
    $d \geq 0$ and $t \geq 1$ where $h_{R/I}(d) = h_{R/I}(d+t) = d+1$.
\end{lemma}
\begin{proof}
    The presence of the strong Lefschetz property clearly implies the second condition.

    Suppose now that the second condition holds.  Let $\varphi_{d,d+t}$ be the map
    $\times (x+y)^t: [R/I]_d \rightarrow [R/I]_{d+t}$, and so
    $\varphi_{a,a+b+c} = \varphi_{a+b,a+b+c} \circ \varphi_{a,a+b}$.

    Theorem~\ref{thm:wlp} implies that $\varphi_{d,d+1}$ always has maximal rank.  In particular,
    $\varphi_{d,d+1}$ is injective for $0 \leq d < \indeg{I}$ and surjective for $\indeg{I} \leq d$.
    Since $\varphi_{d,d+t} = \varphi_{d+t-1,d+t} \circ \cdots \circ \varphi_{d,d+1}$, if
    $d +t \leq \indeg{I}$ (respectively, $d > \indeg{I}$), then each term in the composition
    is injective (respectively, surjective) and so $\varphi_{d,d+t}$ is injective (respectively, surjective).

    Now suppose that $d < \indeg{I} < d+t$.  Since $d+t > \indeg{I}$, then $h(d+t) - 1 \leq \indeg{I}$ and so
    $h(h(d+t) - 1) = h(d+t)$ by Proposition~\ref{pro:classify-hf}.  Thus, by assumption, $\varphi_{h(d+t) - 1, d+t}$
    is a bijection.  If $d < h(d+t) - 1$, then $\varphi_{d,d+t} = \varphi_{h(d+t)-1, d+t} \circ \varphi_{d,h(d+t)-1}$.
    Since $h(d+t)-1 \leq \indeg{I}$, $\varphi_{d,h(d+t)-1}$ is injective.  Hence $\varphi_{d,d+t}$ is injective as
    the composition of injective functions is injective.  On the other hand, if $d > h(d+t) - 1$, then
    $\varphi_{h(d+t) - 1, d+t} = \varphi_{d,d+t} \circ \varphi_{h(d+t)-1, d}$.  Hence $\varphi_{d,d+t}$ is
    surjective as $\varphi_{h(d+t) - 1, d+t}$ is, by assumption, a bijection.
\end{proof}

Let $I$ be a monomial ideal in $R$, and let $d \geq 0$ and $t \geq 1$ be any integers such that
$h(d) = h(d+t) = d+1$.  Suppose the ordered monomials $\{x^{b_0}y^{d+t-b_0}, \ldots, x^{b_d}y^{d+t-b_d}\}$ span $[R/I]_{d+t}$.
Note that the ordering implies that $0 \leq b_0 < \cdots < b_d \leq d+t$.
Then the map $\times(x+y)^t: [R/I]_{d} \longrightarrow [R/I]_{d+t}$ is given by the $(d+1) \times (d+1)$
matrix $N_{R/I}(d,d+t)$, where the entry $(i,j)$ is the coefficient on $x^{b_j}y^{d+t-b_j}$ in $x^{i}y^{d-i}(x+y)^t$,
i.e., $\binom{t}{b_j - i}$.

\begin{example} \label{exa:matrix-N}
    Let $I = (x^{10}, y^{7})$.  Then $h_{R/I}(5) = h_{R/I}(10) = 6$.  Thus $N_{R/I}(5, 5)$ is the
    $6 \times 6$ matrix
    \[
    \left(
    \begin{array}{cccccc}
        5 & 10 & 10 &  5 &  1 &  0 \\
        1 &  5 & 10 & 10 &  5 &  1 \\
        0 &  1 &  5 & 10 & 10 &  5 \\
        0 &  0 &  1 &  5 & 10 & 10 \\
        0 &  0 &  0 &  1 &  5 & 10 \\
        0 &  0 &  0 &  0 &  1 &  5 \\
    \end{array}
    \right).
    \]
    The determinant of this matrix is $210 = 2 \cdot 3 \cdot 5 \cdot 7$, and so the map
    $\times(x+y)^5: [R/I]_{5} \rightarrow [R/I]_{10}$ has maximal rank if and only if
    $\charf{K} = 0$ or $\charf{K} > 7$.
\end{example}

The determinant of $N_{-}(d,d+t)$ can be given by a closed form.

\begin{lemma}\label{lem:square}
    Let $I$ be a monomial ideal in $R$, and let $d \geq 0$ and $t \geq 1$ be integers so that
    $h_{R/I}(d) = h_{R/I}(d+t) = d+1$.  Then $|\det{N_{R/I}(d,d+t)}|$ is
    \[
            \prod_{0 \leq i < j \leq s-r} (b_{r+j} - b_{r+i})
            \prod_{i = 0}^{s-r} \frac{(t+i)!}{(t+s-b_{r+i})! (b_{r+i}-r)!},
    \]
    where the ordered monomials $\{x^{b_0}y^{d+t-b_0}, \ldots, x^{b_{d}}y^{d+t-b_{d}}\}$ span
    $[R/I]_{d+t}$, $r := \max(\{0\} \cup \{k+1 \st b_k = k \})$, and $s := \min(\{d\} \cup \{k-1 \st b_k = t + k \})$.

    If $r = s+1$, then the determinant is one; otherwise, the largest term of the
    above closed form is $w_{R/I}(d+t)-1$, and it appears exactly once.
\end{lemma}
\begin{proof}
    For $0 \leq i < r$, we have $b_i = i$, and so the $(i,j)$ entry of
    $N_{R/I}(d,d+t)$  is $\binom{t}{j-i}$, i.e., $1$ if $i = j$ and $0$ if $i > j$.
    Similarly, for $s < j \leq d$, we have that $b_j = t + j$ and so
    the $(i,j)$ entry of $N_{R/I}(d,d+t)$ is $\binom{t}{t+j-i}$, i.e., is $1$ if
    $i = j$ and $0$ if $i < j$.

    Partitioning $N_{R/I}(d,d+t)$ into a block matrix with square diagonal matrices with
    sizes $r, s-r+1$, and $d - s$, respectively, yields
    \[
        N_{R/I}(d,d+t) = \left(
                    \begin{array}{ccc}
                        U & A       & 0 \\
                        0 & \hat{N} & 0 \\
                        0 & B       & L
                    \end{array}
                \right),
    \]
    where $U$ is a square upper-triangular matrix with ones on the diagonal, $L$ is a square
    lower-triangular matrix with ones on the diagonal, and $\hat{N}$ is an $(s-r+1) \times (s-r+1)$
    square matrix with entry $(i,j)$ given by $\binom{t}{b_{r+j}-(r+i)}$.  Using the block
    matrix formula for the determinant (twice), we have
    \[
        \det{N_{R/I}(d,d+t)} = \det{U} \cdot \det{\hat{N}} \cdot \det{L} = \det{\hat{N}}.
    \]

    Set $A := t$, $n := s-r+1$, and $L_j := b_{r+j-1}-r+1$.  Then the determinant evaluation of $\hat{N}$,
    hence of $N_{R/I}(d,d+t)$, follows immediately from~\cite[Equation~(12.5)]{CEKZ}.  Re-indexing to start from $0$
    instead of $1$ yields the stated result.

    If $r = s+1$, then the products are all empty, hence the determinant is one.  Suppose $r \leq s$.
    Notice that $b_s < t+s$ and $b_{r} > r$.  Hence the largest term in the first product
    $b_{s} - b_{r}$ is less than $t+s-r$, and the largest term in the numerator of the
    second product is $t+s-r$.  The largest term in the denominator of the second product is
    the maximum of $t+s-b_{r}$ and $b_{s}-r$, both of which are less than $t+s-r$.  Hence the
    largest term of the products in the formula is $t+s-r$.

    Notice that by the definitions of $r$ and $s$, $x^r y^{d+t-r}$ and $x^{t+s} y^{d-s}$ are
    the lexicographically smallest and largest monomials in $[I]_{d+t}$, respectively.  Hence
    $w_{R/I}(d+t) = t+s-r+1$ and so the largest term of the products in the formula is $w_{R/I}(d+t)-1$.
\end{proof}

The matrix $\hat{N}$ in the preceding proof is the matrix $N_{R/J}(s-r,s-r+t)$, where $J$ is the ideal
$(x^{b_r-r}y^{t+s-b_r}, \ldots, x^{b_s-r}y^{t+s-b_s}) + (x,y)^{t+s-r+1}$.

\begin{example}\label{exa:N-to-hat-N}
    Let $I = (x^{15}, x^{10}y^2, x^2y^9, y^{15})$.  Then $h_{R/I}(9) = 10 = h_{R/I}(14)$, so $N_{R/I}(9, 14)$
     is the $10 \times 10$ matrix
    \[
    \left(
    \begin{array}{cc|cccccc|cc}
            1& 5& 10&  5&  1&  0&  0&  0& 0& 0  \\
        0& 1& 10& 10&  5&  1&  0&  0& 0& 0  \\ \hline
        0& 0&  5& 10& 10&  5&  1&  0& 0& 0  \\
        0& 0&  1&  5& 10& 10&  5&  1& 0& 0  \\
        0& 0&  0&  1&  5& 10& 10&  5& 0& 0  \\
        0& 0&  0&  0&  1&  5& 10& 10& 0& 0  \\
        0& 0&  0&  0&  0&  1&  5& 10& 0& 0  \\
        0& 0&  0&  0&  0&  0&  1&  5& 0& 0  \\ \hline
        0& 0&  0&  0&  0&  0&  0&  1& 1& 0  \\
        0& 0&  0&  0&  0&  0&  0&  0& 5& 1  \\
    \end{array}
    \right).
    \]
    Notice that the upper-left $2 \times 2$ matrix is upper-triangular, the lower-left
    $2 \times 2$ matrix is lower-triangular, and the central $6 \times 6$ matrix has determinant
    $210 = 2 \cdot 3 \cdot 5 \cdot 7$.  Thus the determinant of $N_{R/I}(9,14)$ has magnitude
    $210$, and so the map $\times(x+y)^5: [R/I]_{9} \rightarrow [R/I]_{14}$ has maximal rank if
    and only if $\charf{K} = 0$ or $\charf{K} > 7$.

    Indeed, the central $6 \times 6$ matrix is the same as the matrix in Example~\ref{exa:matrix-N},
    and is the matrix $\hat{N}$ in the proof of Lemma~\ref{lem:square}.  Furthermore, following
    the preceding remark, we note that $d = 9$, $t = 5$, $r = 2$, and $s = 7$ for $R/I$ in degree $14$.
    Hence $\hat{N}$ is the matrix $N_{R/J}(s-r, s-r+t) = N_{R/J}(5, 10)$, where
    $J = (x^{10}, x^3y^7, x^2y^8, xy^9, y^{10})$.  Notice that $R/J$ and the quotient
    considered in Example~\ref{exa:matrix-N} are the same in degrees $5$ and $10$.
\end{example}

Analysing the formula in Lemma~\ref{lem:square}, we determine exactly when $|\det{N}| = 1$.

\begin{corollary}\label{cor:formula-1}
    Let $I$ be a monomial ideal in $R$, and let $d \geq 0$ and $t \geq 1$ be integers so that
    $h_{R/I}(d) = h_{R/I}(d+t) = d+1$.  Then the following statements are equivalent:
    \begin{enumerate}
        \item $|\det{N_{R/I}(d,d+t)}| = 1$,
        \item $w_{R/I}(d+t) = t$, and
        \item $I$ is lexsegment in degree $d+t$.
    \end{enumerate}

    Specifically, if $t = 1$, then $|\det{N_{R/I}(d,d+t)}| = 1$.
\end{corollary}
\begin{proof}
    By Lemma~\ref{lem:lex-seg-dist}, $w_{R/I}(d+t) = d+t-1 - h_{R/I}(d+t) = t$ if and only if
    $I$ is lexsegment in degree $d$, hence (ii) is equivalent to (iii).

    Suppose $[R/I]_{d+t}$ is spanned by the ordered monomials
    $\{x^{b_0}y^{d+t-b_0}, \ldots, x^{b_{d}}y^{d+t-b_{d}}\}$.  Set
    $r := \max(\{0\} \cup \{k+1 \st b_k = k \})$ and $s := \min(\{d\} \cup \{k-1 \st b_k = t + k \})$.
    Then $r < b_r < \cdots < b_s < s+t$.  By Lemma~\ref{lem:square}, $|\det{N_{R/I}(d,d+t)}|$ is
    \[
            \prod_{0 \leq i < j \leq s-r} (b_{r+j} - b_{r+i})
            \prod_{i = 0}^{s-r} \frac{(t+i)!}{(t+s-b_{r+i})! (b_{r+i}-r)!}.
    \]
    Further, notice that $w_{R/I}(d+t) = s-r+t+1$.

    Now we prove (i) is equivalent to (ii).
    Let $r, s,$ and $t$ be integers such that $0 \leq r \leq s+1$ and $t \geq 2$.
    Let $b_r, \ldots, b_s$ be integers such that $r < b_r < \cdots < b_s < s+t$.
    Define $D(r, s, t, \{b_r, \ldots, b_s\})$ to be
    \[
            \prod_{0 \leq i < j \leq s-r} (b_{r+j} - b_{r+i})
            \prod_{i = 0}^{s-r} \frac{(t+i)!}{(t+s-b_{r+i})! (b_{r+i}-r)!}.
    \]
    Clearly $D(r, s, t, \{b_r, \ldots, b_s\}) \geq 1$ for all valid arguments.

    Notice that $D(r, s, t, \{b_r, \ldots, b_s\}) = 1$ if $r = s + 1$, as the products
    are all empty.  Furthermore, if $t = 1$ and $r \leq s$, then $r < b_r < \cdots < b_s < s+1$ implies that there
    are at least $s-r+1$ distinct integers exclusively between $r$ and $s+1$; however, there are only $s-r$
    such integers.  Hence if $t = 1$, then $r = s+1$, and we can define $D(s+1, s+1, 1, \emptyset) = 1$ for
    every $s$.  (We also note that the case when $t=1$ follows immediately from Theorem~\ref{thm:wlp}.)

    \emph{Step $1$: Base case.} Note that $D(s+1, s+1, t, \emptyset) = 1$ for all $s$ and $t$.
    Further, if $r = s$ and $t \geq 2$, then $D(r, r, t, \{b_r\}) = \binom{t}{b_r-r} \geq t$, as $1 \leq b_r - r \leq t-1$.

    \emph{Step $2$: Increasing $t$.} Assume $r < s$ and $t \geq 1$.  Clearly since $b_s < s+t$, then $b_s<s+t+1$.
    Thus $(r,s,t+1,\{b_r, \ldots, b_s\})$ forms a valid argument for $D(\cdot)$.  Indeed, we can rewrite
    $D(r,s,t+1,\{b_r, \ldots, b_s\})$ in terms of $D(r,s,t,\{b_r, \ldots, b_s\})$ as follows:
    \begin{equation*}
        \begin{split}
            D(r,s,t+1,\{b_r, \ldots, b_s\})
            & = \prod_{0 \leq i < j \leq s-r} (b_{r+j} - b_{r+i}) \prod_{i = 0}^{s-r} \frac{((t+1)+i)!}{((t+1)+s-b_{r+i})! (b_{r+i}-r)!} \\
            & = \prod_{i=0}^{s-r} \frac{((t+1)+i)}{((t+1)+s-b_{r+i})} D(r,s,t,\{b_r, \ldots, b_s\}).
        \end{split}
    \end{equation*}
    For $0 \leq i \leq s-r$, we have $b_{r+i} > r+i$ and so $(t+1)+s-b_{r+i} \leq t+s+1 - (r+i+1) = t+s-r-i$.  Further still,
    $t+1+i = t+s-r-j+1$, where $j = s-r-i$, i.e., $0 \leq j \leq s-r$.  Thus we have
    \[
        \prod_{i=0}^{s-r} (t+1+i) = \prod_{j=0}^{s-r}(t+s-r-j+1) > \prod_{j=0}^{s-r}(t+s-r-j) \geq \prod_{i=0}^{s-r}(t+1+s-b_{r+i}).
    \]
    Hence $\prod_{i=0}^{s-r} \frac{(t+1)+i)}{((t+1)+s-b_{r+i})} > 1$ and $D(r,s,t+1,\{b_r, \ldots, b_s\}) > D(r,s,t,\{b_r, \ldots, b_s\}) \geq 1$.
    Thus, if $t \geq 2$ and $b_s < s+t-1$, then $D(r,s,t\{b_r, \ldots, b_s\}) > 1$.

    \emph{Step $3$: When $b_s = s+t-1$.}  Suppose $r < s$, $t \geq 2$, and $b_s = s+t-1$.
    Since $r \leq s$, then $r \leq (s-1)+1$ and so $(r, s-1, t, \{b_r, \ldots, b_{s-1}\})$ forms a
    valid argument for $D(\cdot)$.  Indeed, we can rewrite $D(r,s,t,\{b_r, \ldots, b_s\})$ in terms
    of $D(r, s-1, t, \{b_r, \ldots, b_{s-1}\})$ as follows:
    \begin{equation*}
        \begin{split}
            & D(r,s,t,\{b_r, \ldots, b_s\}) \\
            =& \prod_{0 \leq i < j \leq s-r} (b_{r+j} - b_{r+i}) \prod_{i = 0}^{s-r} \frac{(t+i)!}{(t+s-b_{r+i})! (b_{r+i}-r)!} \\
            =& \left( \frac{(t+s-r)!}{(t+s-b_s)!(b_s-r)!} \prod_{0 \leq i < s-r} \frac{b_s - b_{r+i}}{t+s-b_{r+i}} \right)
                D(r, s-1, t, \{b_r, \ldots, b_{s-1}\}) \\
            =& \left( (t+s-r) \prod_{0 \leq i < s-r} \frac{t+s-1 - b_{r+i}}{t+s-b_{r+i}} \right) D(r, s-1, t, \{b_r, \ldots, b_{s-1}\}) \\
            =& \frac{(t+s-r)(t+s-1-b_r)}{t+s-b_{s-1}} D(r, s-1, t, \{b_r, \ldots, b_{s-1}\}).
        \end{split}
    \end{equation*}
    Notice though, since $r < b_{s-1}$, we have $t+s-r > t+s-b_{s-1}$ and so $\frac{t+s-r}{t+s-b_{s-1}} > 1$.
    Thus, $D(r,s,t,\{b_r, \ldots, b_s\}) > D(r, s-1, t, \{b_r, \ldots, b_{s-1}\}) \geq 1$.

    \emph{Conclusion.} We thus see that $D(r,s,t,\{b_r, \ldots, b_s\}) = 1$ if and only if $r = s+1$.
    This is equivalent to $w_{R/I}(d+t) = t$, as $w_{R/I}(d+t) = s-r+t+1$.
\end{proof}

\section{The strong Lefschetz property}\label{sec:slp}

Let $I$ be a homogeneous ideal of $S = K[x_1, \ldots, x_n]$, and fix a term order $<$ on the monomials of $S$
(e.g., the lexicographic order).  The initial ideal $\inid_{<}(g \cdot I)$ that is fixed on
some Zariski open subset of $GL_n(K)$ is called the \emph{generic initial ideal} of $I$ with respect to
$<$ and is denoted by $\gin{I} = \gin_<{I}$.

\begin{remark} \label{rem:gin-pos-char}
    The generic initial ideal is known to be strongly stable if the characteristic of $K$ is zero (see, e.g.,
    \cite[Proposition~4.2.6]{HH}).  For ideals in $K[x,y]$, being strongly stable is equivalent to being lexsegment.
    However, the generic initial ideal is not so well-behaved in positive characteristic.  Indeed, \cite[Example~4.2.8]{HH}
    shows that, for every prime $p > 0$, $\gin{(x^p, y^p)} = (x^p, y^p)$ if the characteristic of $K$ is $p$.
\end{remark}

Recall that an artinian ideal has the weak (strong) Lefschetz property exactly when its generic initial ideal
with respect to the reverse lexicographic order has the weak (strong) Lefschetz property.

\begin{proposition}{\cite[Proposition~2.8]{Wi}} \label{pro:slp-gin}
    Let $I$ be a homogeneous artinian ideal in $S = K[x_1, \ldots, x_n]$, and let $J$ be the generic initial
    ideal of $I$ with respect to the reverse lexicographic order.  Then $S/I$ has the weak (strong) Lefschetz
    property if and only if $S/J$ has the weak (strong) Lefschetz property.
\end{proposition}

Further, for $R = K[x,y]$ the lexicographic and reverse lexicographic orders are identical.  In some cases,
we can use the above result to lift statements about monomial ideals to statements about homogeneous ideals.

\subsection{Bounds on the absence of the strong Lefschetz property}\label{sub:bounds}~

We first use the width function to bound the characteristics in which the strong Lefschetz property
can be absent.  See Corollary~\ref{cor:w-sharpness} for a classification of when this bound is sharp.

\begin{theorem}\label{thm:w-bound}
    Let $I$ be a monomial artinian ideal in $R = K[x,y]$, where $K$ is a field of characteristic $p > 0$.
    Then $R/I$ has the strong Lefschetz property if $p \geq w_{R/I}(\reg{R/I})$.
\end{theorem}
\begin{proof}
    Let $d \geq 0$ and $t \geq 1$ be integers such that $h(d) = h(d+t) = d+1$.  Then the largest factor of
    the determinant of $N_{R/I}(d,d+t)$ is $w_{R/I}(d+t) - 1$ by Lemma~\ref{lem:square}.  Hence $N_{R/I}(d,d+t)$
    has maximal rank if $p \geq w_{R/I}(d+t)$.

    Further, we have that $w_{R/I}(\reg{R/I}) \geq w_{R/I}(i)$ for all $0 \leq i \leq \reg{R/I}$ by
    Proposition~\ref{pro:classify-width}.  This implies that every matrix $N_{R/I}(d,d+t)$ has maximal rank if $p \geq w_{R/I}(\reg{R/I})$.
    Therefore, $R/I$ has the strong Lefschetz property by Lemma~\ref{lem:slp-same-hf} if $p \geq w_{R/I}(\reg{R/I})$.
\end{proof}

We now recover Theorem~\ref{thm:slp-char0-reg} with different techniques than those used in~\cite{HMNW}
and~\cite{BI}.  (\emph{Nota bene}: We have not assumed Theorem~\ref{thm:slp-char0-reg} up to this point; doing
so offers no benefit as our approach and desired results are different.)  In particular, by weakening the bound
in the preceding proof we can generalise from monomial ideals to homogeneous ideals.  See
Corollary~\ref{cor:p-failure-reg-sharpness} for a classification of when this bound is sharp.

\begin{theorem} 
    Let $I$ be a homogeneous artinian ideal in $R = K[x,y]$, where $K$ is a field of characteristic $p > 0$.
    Then $R/I$ has the strong Lefschetz property if $p > \reg{R/I}$.
\end{theorem}
\begin{proof}
    Let $J = \gin{I}$ under the (reverse) lexicographic order.  Then by Theorem~\ref{thm:w-bound}, $R/J$ has the strong
    Lefschetz property if $p \geq w_{R/J}(\reg{R/J})$.  By Proposition~\ref{pro:classify-width} we have that
    $w_{R/J}(\reg{R/J}) \leq \reg{R/J} + 1$, and since the generic initial ideal preserves the Hilbert function
    we have $\reg{R/J} = \reg{R/I}$ as $J$ is artinian.  Hence $R/I$ has the strong Lefschetz property if
    $p > \reg{R/I}$ by Proposition~\ref{pro:slp-gin}.
\end{proof}

\begin{remark} \label{rem:reg-bound}
    The bound in the preceding theorem sometimes holds in higher codimension.  Let $I$ be a monomial complete
    intersection in $S = K[x_1, \ldots, x_n]$, i.e., $I = (x_1^{d_1}, \ldots, x_n^{d_n})$.  Then
    we have that $S/I$ has the strong Lefschetz property if $p > \reg{S/I}$ by~\cite[Theorem~3.6(ii)]{Co}.
    Moreover, this bound is sharp in some cases; in particular, if $p = \reg{S/I}$, then $S/I$ fails to have
    the strong Lefschetz property by~\cite[Theorem~3.6(i)]{Co}.  See Corollary~\ref{cor:p-failure-reg-sharpness} for
    a similar statement about the sharpness of the bound in the preceding corollary.

    On the other hand, the bound is not true in general, even as a bound for the failure of the weak Lefschetz property.
    Let $I = (x^{20}, y^{20}, z^{20}, x^3 y^8 z^{13})$ be an ideal of $S = K[x,y,z]$.  In this case, $\reg{S/I} = 50$, and
    $S/I$ has the weak Lefschetz property if and only if the characteristic of $K$ is \emph{not} one of the following primes:
    $2$, $3$, $5$, $7$, $11$, $17$, $19$, $23$, or $20554657$.  This example comes from~\cite{CN}, wherein the
    presence and absence of the weak Lefschetz property for monomial ideals in codimension three is considered.
\end{remark}

We now consider for which characteristics the strong Lefschetz property is absent.

\begin{lemma}\label{lem:w-failure}
    Let $I$ be a monomial artinian ideal in $R = K[x,y]$.  Suppose there exists a $j$
    such that $\indeg{I} \leq j \leq \reg{R/I}$, $w_{R/I}(j) \neq j+1 - h_{R/I}(j)$, and
    $w_{R/I}(j) - 1 = \charf{K}$ is prime.  Then $R/I$ fails to have the strong Lefschetz
    property.
\end{lemma}
\begin{proof}
    Let $d = h_{R/I}(j) - 1$ and $t = j - d$, then $h_{R/I}(d) = d + 1 = h_{R/I}(d+t)$.
    By Lemma~\ref{lem:square}, since $w_{R/I}(d+t) > d+1 - h_{R/I}(d+t)$, we have that
    the largest factor of the formula giving the determinant of $N_{R/I}(d,d+t)$ is
    $w_{R/I}(d+t) - 1 = \charf{K}$.  This implies that the map
    $\times (x+y)^t: [R/I]_{d} \rightarrow [R/I]_{d+t}$ fails to have maximal rank,
    i.e., $R/I$ fails to have the strong Lefschetz property.
\end{proof}

Using this we classify exactly when the bound in Theorem~\ref{thm:w-bound} is sharp.

\begin{corollary}\label{cor:w-sharpness}
    Let $I$ be a monomial artinian ideal in $R = K[x,y]$, and suppose $p = w_{R/I}(\reg{R/I})-1$ is prime.
    Then $R/I$ fails to have the strong Lefschetz property in characteristic $p$ if and only if 
	$x^{\reg{R/I}}, y^{\reg{R/I}} \in I$.
\end{corollary}
\begin{proof}
	Suppose $p = w_{R/I}(\reg{R/I})-1$ is prime.  If one of $x^{\reg{R/I}}$ and $y^{\reg{R/I}}$ is not in $I$,
    then $R/I$ can only fail to have the strong Lefschetz property in characteristics smaller than 
    $p = w_{R/I}(\reg{R/I})-1$ by Theorem~\ref{thm:w-bound}.

    On the other hand, suppose $x^{\reg{R/I}}, y^{\reg{R/I}} \in I$.  Then $w_{R/I}(\reg{R/I}) \neq \reg{R/I}+1 - h_{R/I}(\reg{R/I})$,
    and hence $R/I$ fails to have the strong Lefschetz property by Lemma~\ref{lem:w-failure}.
\end{proof}

As with Theorems~\ref{thm:w-bound} and~\ref{thm:slp-char0-reg}, the above pair of results
can be extended to homogeneous ideals, if we strengthen the restrictions on the ideals.  This
in turn classifies exactly when the bound in Theorem~\ref{thm:slp-char0-reg} is sharp.

\begin{corollary}\label{cor:p-failure-reg-sharpness}
    Let $I$ be a homogeneous artinian ideal in $R = K[x,y]$, where $K$ is a field of characteristic $p > 0$.
    If $p \leq \reg{R/I}$ and $x^p, y^p \in \gin{I}$ (under the lexicographic order), then
    $R/I$ fails to have the strong Lefschetz property.

	In particular, if $\reg{R/I}$ is prime and $p = \reg{R/I}$, then $R/I$ fails to have the strong Lefschetz
	property if and only if $x^p, y^p \in \gin{I}$ (under the lexicographic order).
\end{corollary}
\begin{proof}
    Let $J = \gin{I}$ under the (reverse) lexicographic order.  If $x^p, y^p \in J$, then $w_{R/J}(p) = p+1$.
    Since $p \leq \reg{R/I} = \reg{R/J}$, then $w_{R/J}(p) \neq p+1 - h_{R/J}(p)$.  Hence, by
    Lemma~\ref{lem:w-failure}, $R/J$ fails to have the strong Lefschetz property, and so
    $R/I$ fails to have the strong Lefschetz property by Proposition~\ref{pro:slp-gin}.
\end{proof}

However, the bounds can be far from sharp in some cases.

\begin{example}\label{exa:sharpness}
    Ignoring lexsegment ideals (which always have the strong Lefschetz property by Theorem~\ref{thm:always-slp}),
    the bounds in Theorems~\ref{thm:w-bound} and~\ref{thm:slp-char0-reg}
    are far from sharp in some cases.  In particular, consider the ideal $I_n = (x^{2^n}, y^2)$,
    where $n \geq 1$.  By~\cite[Lemma~4.2(i)]{Co}, we have that $I_n$ has the strong Lefschetz
    property if and only if the characteristic of $K$ divides $2^n$, i.e., $\charf{K} = 2$.
    However, $\reg{R/I_n} = 2^n$, so the regularity bound is sharp if and only if $n = 1$.
    Moreover, $w_{R/I_n}(\reg{R/I_n}) = 2^n+1$, so the width bound is never sharp for this family.
\end{example}

\begin{remark}\label{rem:max-gen-deg}
    Corollary~\ref{cor:p-failure-reg-sharpness} further implies that the maximal degree of a minimal
    generator is not a good bound.  For example,
    let $I_p = (x^{(p+1)/2}, y^{(p+3)/2})$, where $p$ is an odd prime.  Then $\reg{R/I_p} = p$
    and $x^p, y^p \in I_p$, so by Corollary~\ref{cor:p-failure-reg-sharpness} $R/I_p$ fails to
    have the strong Lefschetz property in characteristic $p$.  Notice that the maximal
    generating degree of a minimal generator of $I_p$ is $\frac{p+3}{2}$, which is less
    than $p$ for $p > 3$.
\end{remark}

\subsection{Forcing the presence of the strong Lefschetz property}\label{sub:forcing}~

Let $S = K[x_1, \ldots, x_n]$ be a polynomial ring, where $K$ is a field of characteristic zero.
Then Migliore and Zanello~\cite[Theorem~5]{MZ} classified the Hilbert functions of artinian quotients of
$S$ that force the weak Lefschetz property to hold.  Similarly, Zanello and Zylinski~\cite[Corollary~3.3]{ZZ}
proved all artinian quotients of $S$ with a given Hilbert function have the strong Lefschetz property
if and only if the initial lexsegment ideal with the given Hilbert function has the strong Lefschetz property.

In some cases, the strong Lefschetz property always holds, regardless of characteristic.  The following theorem
classifies the monomial ideals with this property.  (See Proposition~\ref{pro:slp-small-nonmono} for an infinite
family of non-monomial ideals with this property.)

\begin{theorem}\label{thm:always-slp}
    Let $I$ be a monomial artinian ideal in $R = K[x,y]$.  Then $R/I$ always has the strong Lefschetz
    property, regardless of the characteristic of $K$, if and only if $I$ is a lexsegment ideal.
\end{theorem}
\begin{proof}
    Suppose $I$ is not a lexsegment ideal.  Then there is a degree $j \geq \indeg{I}$ such that $I$ is not
    lexsegment in degree $j$.  Let $d = h_{R/I}(j) - 1$ and $t = j - d$, then
    $h_{R/I}(d) = d + 1 = h_{R/I}(d+t)$.  By Corollary~\ref{cor:formula-1} we see that $N_{R/I}(d,d+t)$ does not
    always have maximal rank, hence $R/I$ does not always have the strong Lefschetz property.

    Suppose now that $I$ is a lexsegment ideal.  Then by Corollary~\ref{cor:formula-1} the matrices
    $N_{R/I}(d,d+t)$, where $h(d) = h(d+t) = d+1$, always have maximal rank.  Therefore, $R/I$ always has
    the strong Lefschetz property, regardless of the characteristic of $K$, by Lemma~\ref{lem:slp-same-hf}.
\end{proof}

This implies that no Hilbert function (see Proposition~\ref{pro:classify-hf}) or width function (see Proposition~\ref{pro:classify-width})
can force the strong Lefschetz property to be absent in \emph{some} characteristic for all ideals with the given Hilbert or
width function.  On the other hand, we can describe a large class of Hilbert functions and width functions that force the strong
Lefschetz property to be present.

Using Lemma~\ref{lem:h-force-lexsegment}, we classify the Hilbert functions that force monomial ideals to
\emph{always} have the strong Lefschetz property, regardless of characteristic.

\begin{proposition}\label{pro:h-force-mono-slp}
    Suppose $h: \NN_0 \to \NN_0$ is the Hilbert function of some homogeneous artinian quotient of $R = K[x,y]$.
    Then every monomial quotient of $R$ such that $h_{R/I} = h$ has the strong Lefschetz property, regardless of the
    characteristic of $K$, if and only if for every nonnegative integer $d$ such that $h(d) > h(d+1)$ then $h(d+1) = h(d+2)$.
\end{proposition}
\begin{proof}
    Combine Lemma~\ref{lem:h-force-lexsegment} and Theorem~\ref{thm:always-slp}.
\end{proof}

Hence, we can force homogeneous ideals with these Hilbert functions to have the strong Lefschetz property.

\begin{corollary}\label{thm:h-force-slp}
    Suppose $h: \NN_0 \to \NN_0$ is the Hilbert function of some homogeneous artinian quotient of $R = K[x,y]$.
    If for every nonnegative integer $d$ such that $h(d) > h(d+1)$ then $h(d+1) = h(d+2)$,
    then every homogeneous artinian quotient $R/I$ with $h_{R/I} = h$ has the strong Lefschetz property.
\end{corollary}
\begin{proof}
    Let $R/I$ be some homogeneous artinian quotient with $h_{R/I} = h$.  Since $\gin{I}$ preserves the
    Hilbert function, we have that $h_{R/\gin{I}} = h$.  By Proposition~\ref{pro:h-force-mono-slp}, $R/\gin{I}$
    has the strong Lefschetz property, and so $R/I$ has the strong Lefschetz property by Proposition~\ref{pro:slp-gin}.
\end{proof}

\begin{remark}\label{rem:switching-fields}
    We must be careful with the ring in which we consider an ideal to be.  For example, the ideal $I = (x^2 + y^2, x^3 + y^3)$
    is in $R = K[x,y]$, regardless of the field $K$.  However, $I$ is not artinian in a field of characteristic
    two.  Indeed, if $\charf{K} = 2$, then $(x,y) = (1,1)$ is a non-trivial common zero of the generators of $I$.  However,
    in all other characteristics $R/I$ is artinian and has $h$-vector $(1,2,2,1)$.

    Suppose $\charf{K} \neq 2$.  The reduced Gr\"obner basis for $I$ is $(x^2+y^2, xy^2-y^3, y^4)$, and so
    the initial ideal of $I$ is $\inid{I} = (x^2, xy^2, y^4)$.  Notice that $\inid{I}$ is lexsegment, and so always
    has the strong Lefschetz property by Theorem~\ref{thm:always-slp}.  Using~\cite[Proposition~2.9]{Wi}, the latter
    implies that $R/I$ also has the strong Lefschetz property, if $\charf{K} \neq 2$.
\end{remark}

Moreover, we classify the width functions that force monomial ideals to \emph{always} have the strong Lefschetz property,
regardless of characteristic, using Lemma~\ref{lem:w-force-lexsegment}.

\begin{proposition}\label{pro:w-force-mono-slp}
    Suppose $w: \NN_0 \to \NN_0$ is the width function of some monomial artinian quotient of $R = K[x,y]$.
    Then every monomial artinian quotient $R/I$ such that $w_{R/I} = w$ has the strong Lefschetz property, regardless of the
    characteristic of $K$, if and only if for every nonnegative integer $d$ we have $0 \leq w(d+1) - w(d) \leq 2$.
\end{proposition}
\begin{proof}
    Combine Lemma~\ref{lem:w-force-lexsegment} and Theorem~\ref{thm:always-slp}.
\end{proof}

Hence, we can force homogeneous ideals with these width functions to have the strong Lefschetz property.

\begin{corollary}\label{thm:w-force-slp}
    Suppose $w: \NN_0 \to \NN_0$ is the width function of some homogeneous artinian quotient of $R = K[x,y]$.
    If for every nonnegative integer $d$ we have $0 \leq w(d+1) - w(d) \leq 2$,
    then every homogeneous artinian quotient $R/I$ with $w_{R/\gin{I}} = w$ has the strong Lefschetz property.
\end{corollary}
\begin{proof}
    Let $R/I$ be some homogeneous artinian quotient with $w_{R/\gin{I}} = w$.  By Proposition~\ref{pro:h-force-mono-slp},
    $R/\gin{I}$ has the strong Lefschetz property, and so $R/I$ has the strong Lefschetz property by Proposition~\ref{pro:slp-gin}.
\end{proof}

\section{Observations}\label{sec:obs}

We close with some observations and connections.

\subsection{Complete intersections}\label{sub:ci}~

Let $I = (f_1, \ldots, f_n)$ be a homogeneous ideal in $S = K[x_1,\ldots,x_n]$.  Then $I$ is said to be a
\emph{complete intersection} of type $(\deg{f_1}, \ldots, \deg{f_n})$ if the generators of $I$ form a regular
sequence in $S$.  This is equivalent to $S/I$ being artinian.  We considered the presence of the strong
Lefschetz property in positive characteristic for \emph{monomial} complete intersections in~\cite{Co}.  Here we
show that the two of the results therein do not hold for non-monomial ideals.

The following lemma classifies the presence of the strong Lefschetz property in positive characteristic for
monomial complete intersections of type $(a, b)$, where $2 \leq a \leq 3$ and $a \leq b$.

\begin{lemma}{\cite[Lemma~4.2]{Co}} \label{lem:slp-small-mono}
    Let $R = K[x,y]$ and $p$ be the characteristic of $K$.  Then the following statements both hold.
    \begin{enumerate}
        \item $R/(x^a,y^2)$, for $a \geq 2$, has the strong Lefschetz property if and only if $p$ does not divide $a$.
        \item $R/(x^a,y^3)$, for $a \geq 3$, has the strong Lefschetz property if and only if $p = 2$ and $a\equiv 2 \pmod{4}$ or
            $p \neq 2$ and $a$ is not equivalent to $-1$, $0$, or $1$ modulo $p$
    \end{enumerate}
\end{lemma}

This implies that every monomial complete intersection of type $(a, b)$, where $2 \leq a \leq 3$ and $a \leq b$,
fails to have the strong Lefschetz property for \emph{some} positive characteristic.  (Furthermore, by Theorem~\ref{thm:always-slp}
we see that the only monomial complete intersections in $R$ that\emph{always} have the strong Lefschetz property
are of type $(a, 1)$, where $a \geq 1$.)  This is not true for arbitrary complete intersections of the same types.
Indeed, the following result shows that there exist non-monomial complete intersections of type $(a, b)$, where
$2 \leq a \leq 3$ and $a \leq b$, that \emph{always} have the strong Lefschetz property.

\begin{proposition}\label{pro:slp-small-nonmono}
    Let $R = K[x,y]$.  Then the following non-monomial homogeneous artinian ideals always have the strong Lefschetz
    property, regardless of the characteristic of $K$.
    \begin{enumerate}
    \item $(x^2, xy^{b-1} + y^b)$, where $b \geq 2$, and
    \item $(x^3, x^2y^{b-2} + y^b)$, where $b \geq 3$.
    \end{enumerate}
\end{proposition}
\begin{proof}
    Suppose that $I = (x^2, xy^{b-1} + y^b)$, where $b \geq 2$.  Let $G = \{x^2, xy^{b-1} + y^b, y^{b+1}\}$.
    Clearly $G$ is a reduced Gr\"obner basis for $I$, regardless of the characteristic of $K$.
    So $\inid{I} = (x^2, xy^{b-1}, y^{b+1})$.

    On the other hand, suppose $I = (x^3, x^2y^{b-2} + y^b)$, where $b \geq 3$.  Let $G = \{x^3, x^2y^{b-2} + y^b, xy^b, y^{b+2}\}$.
    Clearly $G$ is a reduced Gr\"obner basis for $I$, regardless of the characteristic of $K$.
    So $\inid{I} = (x^3, x^2y^{b-2}, xy^b, y^{b+2})$.

    In either case, $\inid{I}$ is an artinian lexsegment ideal, and thus $R/\inid{I}$ always has the strong
    Lefschetz property by Theorem~\ref{thm:always-slp}.  Thus using~\cite[Proposition~2.9]{Wi}, we see that $R/I$
    always has the strong Lefschetz property.
\end{proof}

The following theorem characterises the presence of the strong Lefschetz property in positive characteristic
for monomial complete intersections of type $(d,d)$.

\begin{theorem}{\cite[Theorem~4.9]{Co}} \label{thm:slp-dd}
    Let $R = K[x,y]$, where $p$ is the characteristic of $K$, and $I_d = (x^d, y^d)$, where $d \geq 2$.  Then $R/I_d$
    has the strong Lefschetz property if and only if $p = 0$ or $2d-2 < p^s$, where $s$ is the largest integer such
    that $p^{s-1}$ divides $(2d-1)(2d+1)$.
\end{theorem}

This result implies that $I = (x^p, y^p)$, where $p$ is prime, fails to have the strong Lefschetz property
in characteristic $p$.  However, Proposition~\ref{pro:slp-small-nonmono} implies that this result does
not hold in general for non-monomial complete intersections of type $(2,2)$ and $(3,3)$.  We further checked this for
non-monomial complete intersections of type $(p,p)$, where $5 \leq p \leq 41$ is prime.

\begin{conjecture}\label{con:slp-dd-fails}
    Let $p$ be an odd prime.  Suppose $I = (x^p, x^{(p+1)/2}y^{(p-1)/2} + y^p)$ is an ideal of $R = K[x,y]$,
    where $K$ is a field of characteristic $p$.  Then $R/I$ has the strong Lefschetz property.
\end{conjecture}

\begin{remark}\label{rem:slp-dd-fails}
    Clearly the set $G = \{x^p, x^{(p+1)/2}y^{(p-1)/2} + y^p, x^{(p-1)/2}y^p, y^{(3p+1)/2}\}$ is
    a reduced Gr\"obner basis for $I$ as in the preceding conjecture, regardless of the characteristic of
    $K$.  Hence $\inid{I} = (x^p, x^{(p+1)/2}y^{(p-1)/2}, x^{(p-1)/2}y^p, y^{(3p+1)/2})$, and so $\inid{I}$
    is lexsegment if and only if $p = 3$.
\end{remark}

From the examples in Proposition~\ref{pro:slp-small-nonmono} and Conjecture~\ref{con:slp-dd-fails}, if true,
we see that the absence of the strong Lefschetz property for \emph{monomial} complete intersections of a
fixed type does not imply the same for \emph{non-monomial} complete intersections of the same type.

\subsection{Families of non-intersecting lattice paths}\label{sub:nilps}~

Let $I$ be a monomial artinian ideal in $R = K[x,y]$, and let $d$ and $t$ be integers such that
$0 \leq d < \indeg{I}$ and $\indeg{I} \leq d+t < \reg{R/I}$.  The matrix $N_{R/I}(d,d+t)$ can be
identified with a matrix involving lattice paths in a particular sub-lattice of $\ZZ^2$.  Moreover,
if $N_{R/I}(d,d+t)$ is square, then its determinant provides a meaningful interpretation.  We first
must make several definitions.

A \emph{lattice path} in a sub-lattice $L \subset \ZZ^2$ is a finite sequence of vertices of $L$ so
that all single steps move either to the right (horizontal) or up (vertical). Given any two vertices
$A = (u,v), E = (x,y) \in \ZZ^2$, the number of lattice paths in $\ZZ^2$ from $A$ to $E$ is the
binomial coefficient $\binom{x-u+y-v}{x-u}$, as each lattice path has precisely $x-u+y-v$ steps and
$x-u$ of these must be horizontal steps.  (Note that if $E$ is left of or below $A$, then the number of
lattice paths is zero.)

Let $L_{R/I}$ be the finite sub-lattice of $\ZZ^2$ consisting of the point $(i,j)$, where $x^i y^j$
is not in $I$, that is, $x^i y^j$ is non-zero in $R/I$.  We denote by $L_{R/I}(d,d+t)$ the sub-lattice
$L_{R/I}$ with two sets of distinguished vertices:
\begin{enumerate}
    \item the vertices of $L_{R/I}$ along the line $x+y=d$ are labeled $A_i = (a_i, d - a_i)$,
        where $0 \leq i \leq m$ and $a_0 < \cdots < a_m$, and
    \item the vertices of $L_{R/I}$ along the line $x+y=d+t$ are labeled $E_j = (b_j, d+t - b_j)$,
        where $0 \leq j \leq n$ and $b_0 < \cdots < b_n$.
\end{enumerate}
By careful observation we notice that the distinguished vertices correspond precisely to the monomials
spanning $[R/I]_{d}$ and $[R/I]_{d+t}$, respectively.

The \emph{lattice path matrix} of $L_{R/I}(d,d+t)$ is the $(m+1) \times (n+1)$ matrix $N = N(L_{R/I}(d,d+t))$
with entries $N_{(i,j)}$ defined to be the number of lattice paths in $\ZZ^2$ from $A_i$ to $E_j$, i.e.,
$\binom{b_j - a_i + (d+t - b_j) - (d - a_i)}{b_j - a_i} = \binom{t}{b_j - a_i}$.   Clearly then,
$N(L_{R/I}(d,d+t)) = N_{R/I}(d,d+t)$.

A \emph{family of non-intersecting lattice paths} is a finite collection of lattice paths such that no
two lattice paths in the family have common points.  Moreover, when $m = n$ the matrix
$N(L_{R/I}(d,d+t)) = N_{R/I}(d,d+t)$ is square, and its determinant is given in Lemma~\ref{lem:square}.
We use a theorem first given by Lindstr\"om in~\cite[Lemma~1]{Li}
and stated independently in~\cite[Theorem~1]{GV} by Gessel and Viennot to interpret this determinant.

\begin{theorem}{\cite[Lemma~1]{Li} \& \cite[Theorem~1]{GV}} \label{thm:lgv}
    Let $A_1, \ldots, A_m$ and $E_1, \ldots, E_m$ be distinguished vertices of $\ZZ^2$.  Define the $m \times m$
    matrix $N$ to have entry $N_{(i,j)}$ defined to be the number of lattice paths in $\ZZ^2$ from $A_i$ to $E_j$.
    Then
        \[
            \det{N} = \sum_{\lambda \in \PS_m} \sgn(\lambda) P^+_\lambda(A\rightarrow E)
        \]
        where, for each permutation $\lambda \in \PS_m$, $P^+_\lambda(A \rightarrow E)$ is the number of families of
        non-intersecting lattice paths going from $A_i$ to $E_{\lambda(i)}$.
\end{theorem}

\begin{example}\label{exa:lattice}
    Recall that in Example~\ref{exa:matrix-N}, the matrix $N_{R/I}(d,d+t)$ was given, where
    $I = (x^{10}, y^7)$ and $d = t = 5$.  Figure~\ref{fig:lattice} shows the lattice
    $L_{R/I}(d,d+t)$ with the distinguished vertices marked and a family of non-intersecting
    lattice paths in grey.  By Example~\ref{exa:matrix-N}, the determinant of $N_{R/I}(d,d+t)$
    is $210$, so there are $210$ families of non-intersecting lattice paths from $A_i$ to $E_i$
    in $L_{R/I}(d,d+t)$.

    \begin{figure}[!ht]
        \includegraphics[scale=0.85]{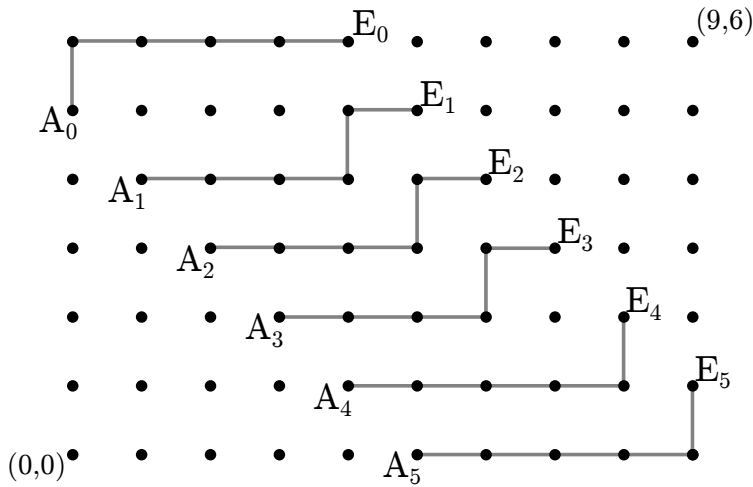}
        \caption{The lattice $L_{R/(x^{10},y^7)}(5,10)$ with distinguished vertices marked
            and a family of non-intersecting lattice paths shown in grey.}
        \label{fig:lattice}
    \end{figure}
\end{example}

Notice that when $m = n$, i.e., the number of $A_i$ vertices is the same as the number of $E_j$ vertices, then
the path starting at $A_i$ in a family of non-intersecting lattice paths must end at $E_i$, for $0 \leq i \leq m$.
Hence by Theorem~\ref{thm:lgv}, $\det{N(L_{R/I}(d,d+t))}$ is the number of non-intersecting lattice paths
in $L_{R/I}(d,d+t)$.

\subsection{Connections to the weak Lefschetz property in codimension three}\label{sub:codim3}~

The connection between the strong Lefschetz property in codimension two and families of non-intersecting
lattice paths described in Section~\ref{sub:nilps} is similar to the connection of the latter to the
weak Lefschetz property in codimension \emph{three}, as described in~\cite{CN}.  Thus we see that the
strong Lefschetz property in codimension two is intimately related to the weak Lefschetz property in
codimension three.

In particular, let $I$ be a monomial artinian ideal in $R = K[x,y]$.  We see that the multiplication map
$\times(x+y)^t: [R/I]_d \rightarrow [R/I]_{d+t}$ has maximal rank exactly when the multiplication map
$\times(x+y+z): [S/J]_{d+t-1} \rightarrow [S/J]_{d+t}$ has maximal rank, where $J = I + (z^t)$ is an ideal of
$S = K[x,y,z]$ (where we abuse notation to see $I$ as a monomial ideal of $S$).  This follows as their cokernels,
$[R/(I, (x+y)^t)]_{d+t}$ and $[S/(J, x+y+z)]_{d+t}$, are isomorphic.
This connection has been used more generally; see, e.g., \cite{Ch}, \cite{Co}, \cite{MM}, and \cite{MMN-2}.

Using the above connection and Theorem~\ref{thm:slp-char0-reg}, we can bound the failure of the weak Lefschetz property
for certain codimension three monomial ideals.

\begin{proposition}\label{pro:wlp-3var}
    Let $J$ be a monomial artinian ideal in $S = K[x,y,z]$, where exactly one generator of $J$
    is divisible by $z$, up to a change of variables.  Then $S/J$ has the weak Lefschetz property
    if the characteristic of $K$ is at least $\reg{S/J}$.
\end{proposition}
\begin{proof}
    Since $J$ is artinian and has exactly one generator divisible by $z$, then that generator must
    be $z^t$ for some positive $t$.  Moreover, the remaining generators of $J$ must be monomials
    in $x$ and $y$ only; let $I$ be the ideal generated by these monomials.  Then $I$ can be
    seen as a monomial artinian ideal in $R = K[x,y]$.

    By Theorem~\ref{thm:slp-char0-reg}, $R/I$ has the strong Lefschetz property if $\charf{K} \geq \reg{R/I}$.
    Moreover, $\reg{S/J} \geq \reg{R/I}$, so $R/I$ has the strong Lefschetz property if
    $\charf{K} \geq \reg{R/J}$.  Thus we see that $S/J$ has the weak Lefschetz property if the
    characteristic of $K$ is at least $\reg{S/J}$.
\end{proof}

\begin{acknowledgement}
    The author would like to thank Juan Migliore and Uwe Nagel for many discussions related to
    the results contained herein and for help in the presentation of this manuscript.

    Further, the author would like to thank Tony Iarrobino for pointing out the earlier results
    that give Theorem~\ref{thm:slp-char0-reg}.
\end{acknowledgement}


\end{document}